\documentclass[11pt]{amsart}
\usepackage[colorlinks]{hyperref}
\hypersetup{nesting=true,debug=true,naturalnames=true}
\usepackage{graphicx,amsmath,amsthm,amscd,amssymb,upref,amsfonts,enumerate}
\usepackage{color}
\newtheorem{thm}{Theorem}[section]

\newtheorem{prop}[thm]{Proposition}
\theoremstyle{definition}

\theoremstyle{remark}

\numberwithin{equation}{section}

\newcommand{\OO}{\mathcal{O}}

\def\Q{\mathbb{Q}}

\newfont{\wncyr}{wncyr10 at 12pt}
\newcommand{\D}{\displaystyle}

\begin{document}
\title
{Generators for the elliptic curve\\\vskip 3mm $E_{(p,q)}: y^2=x^3-p^2x+q^2$}
\author{\bf M. Khazali ,H. Daghigh and A. Alidadi }
\address{Department of Mathematics, Bam Higher Educational Complex Of I. R. Iran}
\email{mehrdad@bam.ac.ir}
\email{hassan@kashanu.ac.ir}
\email{alidadi@bam.ac.ir}


\begin{abstract}

Let $ \lbrace E_{(p,q)} \rbrace $ represents a family of elliptic curves over $ \Q$ as 
defined by the Weierstrass equation  $  E_{(p,q)} : y^2=x^3-p^2x+q^2 $, where p and q are prime numbers exceeding five. 
Previous research established that the   $ \lbrace E_{(p,q)} \rbrace  $ had a rank of at least two for all $p, q > 5 $ and two independent points. 
This research implies that the two points can be extended to a basis for the family $\lbrace E_{(p,q)} \rbrace  $.
\end{abstract}

\maketitle
\vskip 15mm
\section{Introduction}

Let $ \lbrace E_{(1,m)} \rbrace $ represents a family of elliptic curves over $ \Q$ as 
defined by the Weierstrass equation  $ E_{(1,m)}: y^2=x^3-x+m^2 $, where m is a integer  number. 
Brown and Myers discovered in \cite{B} that the this family  included two independent points, Fujita and Nara illustrated in \cite{Fn} that points could be extended to a basis for this family. 

Let $ \lbrace E_{(n,1)} \rbrace $ represents a family of elliptic curves over $ \Q$ as 
defined by the Weierstrass equation  $  E_{(n,1)} : y^2=x^3-n^2x+1 $, where n is a integer  number. 
 Antoniewicz proved in \cite{A}that the this family included two independent points. Fujita and Nara illustrated in \cite{Fn} that points could be extended to a basis for this family. 
 
Let $ \lbrace E_{(m,p)} \rbrace $ represents a family of elliptic curves over $ \Q$ as 
defined by the Weierstrass equation  $ E_{(m,p)} : y^2=x^3-m^2x+p^2 $, where p is a prime number and $m\equiv 2 \pmod {32}$.
 Abhishek and Datt Kumar's demonstrated in   \cite{JU}, the  this family  had two independent points.
 
Let $ \lbrace E_{(p,q)} \rbrace $ represents a family of elliptic curves over $ \Q$ as 
defined by the Weierstrass equation  $  E_{(p,q)} : y^2=x^3-p^2x+q^2 $, where p and q are prime numbers exceeding five. 
  The family $ \lbrace E_{(p,q)} \rbrace$ had a minimum rank of two for all  $p, q > 5 $ and $P_1=(0, q)\in  E_{(p,q)}(\Q) $ and $P_2=(-p, q)\in  E_{(p,q)}(\Q) $ are independent points, as shown by previous research \cite{kh}. 
We describe how the two points  $P_1$ and $P_2$ can be extended to a basis for this family and improved under special conditions in this article. The single most powerful statement is demonstrated by the theorem  \ref{T1}.

\begin{thm}\label{T1}[Main Theorem].
Let $ \lbrace E_{(p,q)} \rbrace $ represents a family of elliptic curves over $ \Q$ as 
defined by the Weierstrass equation  $  E_{(p,q)} : y^2=x^3-p^2x+q^2 $, where p and q are prime numbers exceeding five.
  If  $p> 2\sqrt[4]{2}q$, then  $P_1=(0, q)\in  E_{(p,q)}(\Q)  $ and $P_2=(-p, q)\in  E_{(p,q)}(\Q) $ can be extended to a basis for this family.
\end{thm}

\section{Upper and Lower bound }

In this section, we will continue our investigation of the concept of canonical height, an imperative tool for elliptic curve arithmetic. 
The canonical height of P  by
\[
\begin{array}{cc}
\hat{h}:E(\Q) \longrightarrow [0,\infty)\qquad \qquad \qquad\\
P\longmapsto
\begin{cases}
\D{\lim_{n\to \infty}\frac{h(2^{n}P)}{4^n}} & P\not=\OO.\\
0&P=\OO.
\end{cases}
\end{array}
\]
dose not suitable for computation.Tate's height is the alternative definition of canonical height presented here \cite{S4}.Therefore, we  have 
\[\hat{h}(P)=\hat{\lambda}_{\infty}(P)+\sum_{r|\Delta}\hat{\lambda}_{r}(P).\]

In fact, the canonical height is equal to the sum of the archimedean local height and the local height, assuming that r is a prime number such that $r \mid \Delta$.
Additionally, we observe that the discriminant of $ E (p,q)$ is $ \Delta=16(4p^6-27q^4)=16 \Delta' $.
In a past research  \cite{kh}, we demonstrated that 3 and 5 $\nmid  \Delta' $, but   $ \Delta' $ has no squares in this article. 
Currently, we assert that equation $y^2=x^3-p^2x+q^2$  is the global minimal.

\begin{prop}\label{T12}
The Weierstrass equation  $y^2=x^3-p^2x+q^2$  is the global minimum.
\end{prop}
\begin{proof}
In view of Lemma 3.1 of \cite{Fn} .
\end{proof}

Now, we compute $c_4=48p^2$, $ c_6= -864q^2$, $b_2=0$, $b_4=-2p^2$, $b_6=4q^2$ and $b_8=-p^4$.
The following theorems determine the upper and lower limits for  $P_1$ and $P_2$ canonical height. 

\begin{thm}\label{T2}
Let $ \lbrace E_{(p,q)} \rbrace $ represents a family of elliptic curves over $ \Q$ as 
defined by the Weierstrass equation  $  E_{(p,q)} : y^2=x^3-p^2x+q^2 $, where p and q are prime numbers exceeding five. 
we consider  $P_1=(0, q)\in  E_{(p,q)}(\Q)  $ and $P_2=(-p, q)\in  E_{(p,q)}(\Q) $. 
If $p> 2\sqrt[4]{2}q$, then 
\[\hat{h}(P_1) \leqslant \frac{1}{2}log(p)+\frac{1}{24}log(2^{11}p^4), \quad \hat{h}(P_2) \leqslant \frac{1}{2}log(p)+\frac{1}{6}log(2^{11}p^4). \] 
\end{thm}

\begin{proof}
According to (4.1) of \cite{S4}, we have
\[H=Max\lbrace 4,2p^2,8q^2,p^4\rbrace.\]
 The assumption of the theorem  leads to the conclusion that $H=p^4$. To obtian the upper bound for canonical height for point $P_1$  based on Theorem (2.2) of \cite{S4}, we must use equation \ref{Ee1}.
\begin{equation}\label{Ee1}
\hat{\lambda}_{\infty}(P) =\frac{1}{8}\log(|(x^2+p^2)^2-8q^2x|)+\frac{1}{8}\sum_{n=1}^{\infty}4^{-n} \log(|z(2^{n}P)|).
\end{equation}
Hence, we have
\[\hat{\lambda}_{\infty}(P_1) \leqslant \frac{1}{2}log(p)+\frac{1}{24}log(2^{11}p^4)=UB1,\] and so for point $P_2$
according to Theorem (2.2) of \cite{S4}, we must use the equation \ref{Ee2}.
\begin{equation}\label{Ee2}
\hat{\lambda}_{\infty}(P) =\frac{1}{2}\log(|x|)+\frac{1}{8}\sum_{n=0}^{\infty}4^{-n} \log(|z(2^{n}P)|)
\end{equation}
Hence, we have
\[\hat{\lambda}_{\infty}(P_2) \leqslant \frac{1}{2}log(p)+\frac{1}{6}log(2^{11}p^4)=UB2. \] 
\end{proof}

\begin{thm}\label{T3}
Let $ \lbrace E_{(p,q)} \rbrace $ represents a family of elliptic curves over $ \Q$ as 
defined by the Weierstrass equation  $  E_{(p,q)} : y^2=x^3-p^2x+q^2 $, where p and q are prime numbers exceeding five. Let $P \in  E_{(p,q)}(\Q) $ be  a rational point on $E_{(p,q)}$.
If  $p> 2\sqrt[4]{2}q$, then
\[\hat{h}(P)> \frac{1}{8}log(\frac{p^4}{2})-\frac{1}{3}log(2)=LB. \]
\end{thm}

\begin{proof}
We have two cases for computing the local height based on proposition \ref{T12}and theorem (5.2) of \cite {S4}.
If $P$ reduces to a nonsingular point in module two, then  $\lambda_2(P)=0$.  Otherwise, $P$ becomes a singular point modulo two. According to (c) of the theorem (5.2) of \cite{S4}, we have $\lambda_2(P)=-\frac{1}{3}log(2)$. Next, we show that 
\[\hat{\lambda}_{\infty}(P)\geqslant \frac{1}{8}\log(|(x^2+p^2)^2-8q^2x|)\geqslant \frac{1}{8}\log(|p^4-16q^4|)>\frac{1}{8}log(\frac{p^4}{2}),\] therefore
\[\hat{h}(P)> \frac{1}{8}log(\frac{p^4}{2})-\frac{1}{3}log(2). \]
\end{proof}

\section{Proof of Theorem 1.1 }

The theorem (3.1) of \cite{si} is a fundamental theorem that is used to prove the theorem \ref{T4}.

\begin{thm}\label{T4}
Assume that $E$ is an elliptic curve with a rank of $r\geqslant 2$ over $\Q$. Let $P'_1$ and $P'_2 $ be independent points in  the $E(\Q)$ modulo
$E(\Q)_{tors}$. Choose a basis $\lbrace Q_1,Q_2,\ldots, Q_r \rbrace$ for $E(\Q)$ modulo $E(\Q)_{tors}$
according to the condition   $P'_1,P'_2 \in \langle Q_1 \rangle+\langle Q_2 \rangle $. Assume that $E(\Q)$  
contains no infinite-order point $Q$ with  $\hat{h}(Q)\leqslant \lambda$, where $ \lambda$ is  positive real number. Then, index $v$ of the span of $P'_1$ and
$P'_2$ in $\langle Q_1 \rangle+\langle Q_2 \rangle $ satisfies
\[v\leqslant \frac{2}{\sqrt{3}}\frac{\sqrt{R(P'_1,P'_2)}}{\lambda},\]
where
\[R(P'_1,P'_2)=\hat{h}(P'_1)\hat{h}(P'_2)-\frac{1}{4}(\hat{h}(P'_1+P'_2) -\hat{h}(P'_1) -\hat{h}(P'_2) )^2<\hat{h}(P'_1)\hat{h}(P'_2),\]
thus
\[v\leqslant \frac{2}{\sqrt{3}}\frac{\sqrt{\hat{h}(P'_1)\hat{h}(P'_2)}}{\lambda}.\]
\end{thm}

This has led us to the proof of Theorem  \ref{T1}.
\begin{proof}
In addition to the fact that $2\nmid v$ holds true, we use three theorems to support our claim:   \ref{T2}, \ref{T3} and  \ref{T4}.
Now, we determine the right-hand side of the following equation:
\[v\leqslant \frac{2}{\sqrt{3}}.\frac{\sqrt{ UB1.UB2 }}{LB}.\]
The result of the calculation for all prime numbers  $p\geqslant 41$ is  $v< 3$.  The evidence is therefore persuasive.
\end{proof}

\end{document}